\documentclass[a4paper,11pt]{amsart}

\usepackage{mathpazo}
\usepackage{amssymb}
\usepackage[pagebackref,
    ,pdfborder={0 0 0}
    ,urlcolor=black,a4paper,hypertexnames=false]{hyperref}
\hypersetup{pdfauthor=Clara L\"oh,pdftitle=Odd manifolds of small integral simplicial volume}

\usepackage{pgf}
\usepackage{tikz}
\usetikzlibrary{decorations.pathmorphing,calc}
\usepackage[all]{xy}
\SelectTips{cm}{}
\usepackage{pifont}

\newtheorem{thm}{Theorem}[section]
\newtheorem{prop}[thm]{Proposition}
\newtheorem{lem}[thm]{Lemma}
\newtheorem{cor}[thm]{Corollary}

\newtheorem{question}[thm]{Question}

\theoremstyle{remark}
\newtheorem{rem}[thm]{Remark}

\theoremstyle{definition}
\newtheorem{defi}[thm]{Definition}

\usepackage{amsfonts}
\newcommand{\Z}{\mathbb{Z}}
\newcommand{\Q}{\mathbb{Q}}
\newcommand{\R}{\mathbb{R}}
\newcommand{\N}{\mathbb{N}}
\newcommand{\C}{\mathbb{C}}

\DeclareMathOperator{\map}{map}

\def\epsilon{\varepsilon}

\DeclareMathOperator{\Top}{{\sf Top}}

\def\args{\;\cdot\;}

\def\gvertex#1{%
  \fill #1 circle(0.1);
}

\def\isv#1{%
  \| #1 \|_{\Z}}
\def\ilone#1{%
  \| #1 \|_{1,\Z}}
\def\sv#1{%
  \| #1 \|}

\def\fcl#1{%
  [#1]_\Z}

\def\fclr#1{%
  [#1]_R}

\def\fface#1#2{%
  {#2}\rfloor_{#1}%
}
\def\bface#1#2{%
  {}_{#1}\lfloor{#2}%
}

\usepackage{color}
\usepackage{pdfcolmk}

\long\def\forget#1{}

\def\draftinfo{}

\author{Clara L\"oh}
\title{Odd manifolds of small integral simplicial volume}
\date{\today.\ \copyright{\ C.~L\"oh 2015}. 
    This work was supported by the CRC~1085 \emph{Higher Invariants} 
    (Universit\"at Regensburg, funded by the DFG).
    \draftinfo\\
     MSC~2010 classification: 55N10, 57N65}
\begin{document}

\begin{abstract}
 Integral simplicial volume is a homotopy invariant of oriented closed
 connected manifolds, defined as the minimal weighted number of
 singular simplices needed to represent the fundamental class with
 integral coefficients.  We show that odd-dimensional spheres are the
 only manifolds with integral simplicial volume equal
 to~$1$. Consequently, we obtain an elementary proof that, in general,
 the integral simplicial volume of (triangulated) manifolds is
 \emph{not} computable.  
\end{abstract}

\phantom{.}
\vspace{-.1\baselineskip}

\maketitle

\section{Introduction}

Simplicial volumes are homotopy invariants for oriented closed
connected manifolds, defined as the infimal weighted number of
singular simplices needed to represent the fundamental class with
respect to the given coefficients. Originally, Gromov introduced the
simplicial volume with \mbox{$\R$-co}\-efficients~\cite{munkholm,vbc,mapsimvol}
in his study of Mostow rigidity.  In the context of the connection of
simplicial volume with Betti numbers or $L^2$-Betti numbers, integral
versions of simplicial volume play a central
role~\cite{gromov,mschmidt,loehpagliantini}.

We will consider the most basic integral version: The
\emph{integral simplicial volume} of an oriented closed connected
$n$-manifold~$M$ is defined as
\[ \isv M := \min \bigl\{ |c|_1
   \bigm| \text{$c \in C_n(M;\Z)$ is a $\Z$-fundamental cycle of~$M$}
   \bigr\} 
   \in \N,
\] 
where $|c|_1 := \sum_{j=0}^k |a_j|$ for any chain~$c = \sum_{j=0}^k
a_j \cdot \sigma_j$ in~$C_*(M;\Z)$ (in reduced form).  Integral
simplicial volume lacks most of the convenient inheritance properties
of ordinary simplicial volume (such as multiplicativity under finite
coverings and gluing formulae), and hence is rather difficult to
control.

In the following, we will characterise manifolds with integral simplicial
volume equal to~$1$, and derive an elementary non-computability statement 
for integral simplicial volume:

\begin{thm}
  \label{thm:odd1}
  Let $n \in \N_{>0}$ and let $M$ be an oriented closed connected
  $n$-manifold.  Then $\isv M =1$ if and only if $n$ is odd and $M$ is homotopy
  equivalent to~$S^n$.
\end{thm}

\begin{rem}
  In view of the positive solution of the Poincar\'e conjecture,
  the conditions in Theorem~\ref{thm:odd1} are also equivalent to 
  $M \cong_{\Top} S^n$.
\end{rem}

\begin{cor}\label{cor:noncomp}
  Let $n \in \N_{\geq 5}$ be odd. 
  \begin{enumerate}
    \item It is not decidable whether a given triangulated oriented closed 
      connected $n$-manifold~$M$ satisfies~$\isv M =1$.
    \item In general, the integral simplicial volume of triangulated
      oriented closed connected $n$-manifolds is not computable.
  \end{enumerate}
\end{cor}

The basic strategy of the proof of Theorem~\ref{thm:odd1} is as follows:
\begin{itemize}
  \item We show that any manifold of integral simplicial volume~$1$ of
    dimension~$> 1$ is simply connected. This step is based on
    quotients of the standard simplex that model singular \emph{cycles}
    consisting of a single simplex (Section~\ref{sec:model}).
  \item Similar to the Betti number estimates for integral simplicial
    volume~\cite[Example~14.28]{lueck}, we will use Poincar\'e duality and an
    explicit formula for the cap-product to prove that any manifold of
    integral simplicial volume equal to~$1$ is an integral homology sphere 
    (Lemma~\ref{lem:oddhomology}).
  \item Finally, we apply the Hurewicz and Whitehead theorems to conclude.
\end{itemize}

More generally, this technique also gives a characterisation for
multiples of the fundamental class (Theorem~\ref{thm:odd1mult}), which
is interesting in the context of integral approximations of simplicial
volume (Section~\ref{subsec:intapprox}).

We then derive the non-computability statement of
Corollary~\ref{cor:noncomp} from Theorem~\ref{thm:odd1} by comparison
with the sphere-recognition problem. Alternatively, this result can
also be obtained from Weinberger's non-computability result for
ordinary simplicial volume of homology spheres, which is based on far
less elementary techniques~\cite[Theorem~2, p.~88]{weinberger}.


More concretely, already the case of integral simplicial volume equal
to~$2$ is more involved than Theorem~\ref{thm:odd1}: For example, in
even dimensions, spheres are \emph{not} the only oriented closed
connected manifolds of integral simplicial volume equal to~$2$
(Proposition~\ref{prop:isvs1}). Moreover, the oriented closed connected 
$3$-manifolds with integral simplicial volume equal to~$2$ are precisely 
(up to homotopy/homeomorphism) $S^1 \times S^2$, $\R P^3$, and $L(3,1)$ 
(Proposition~\ref{prop:isv32}, Remark~\ref{rem:32}).

\subsection*{Organisation of this article}
In Section~\ref{sec:model}, we study the combinatorics of singular
cycles that consist of a single singular simplex. Section~\ref{sec:1}
contains a proof of Theorem~\ref{thm:odd1} and a generalisation to
multiples of the fundamental class; moreover, we discuss the relation
with integral approximations of simplicial volume. In
Section~\ref{sec:2}, we discuss some aspects of the case that the
integral simplicial volume equals~$2$. Finally, in
Section~\ref{sec:noncomp} we briefly introduce the relevant setup from
computability theory and derive Corollary~\ref{cor:noncomp}.

\section{Model spaces of cycles of a single singular simplex}
\label{sec:model}

Singular cycles that consist of a single singular simplex have a very
restricted combinatorial structure. The corresponding model spaces for
such cycles are simply connected (Proposition~\ref{prop:modelpi1}),
but fail to be closed manifolds in general
(Proposition~\ref{prop:modelmfd}).

\subsection{The model spaces}

A model space associated with a cycle consisting of a single simplex
is a standard simplex where the faces are glued according to their
cancellation in the singular boundary.

\begin{defi}[matching, model space]
  Let $n \in \N$ be odd. 
  \begin{itemize}
    \item An \emph{$n$-matching} is a bijection~$\{0,2,\dots, n-1\} 
      \longrightarrow \{1,3,\dots,n\}$.
    \item For $j \in \{0,\dots,n\}$, we denote the inclusion of the $j$-th 
      face of~$\Delta^n$ by~$i_j \colon \Delta^{n-1} \longrightarrow \Delta^n$.
    \item Let $\pi$ be an $n$-matching. Then the \emph{model 
      space associated with~$\pi$} is 
      \[ M_\pi := \Delta^n / \sim_\pi, 
      \]
      where $\sim_\pi$ is the equivalence relation generated by
      \[  \forall_{j \in \{0,2,\dots, n-1\}} \forall_{t \in \Delta^{n-1}}
          \quad
           i_j(t) \sim_\pi i_{\pi(j)}(t).
      \]
      We denote the canonical projection~$\Delta^n \longrightarrow M_\pi$
      by~$\sigma_\pi$. 
  \end{itemize}
\end{defi}

\begin{rem}[model class]
  Let $n \in \N$ be odd and let $\pi$ be an $n$-matching. Then 
  \begin{align*}
    \partial_n \sigma_\pi 
    & = \sum_{j=0}^n (-1)^j \cdot \sigma_\pi \circ i_j 
      = \sum_{j = 0}^{(n-1)/2} ( \sigma_\pi \circ i_{2\cdot j} 
      - 
      \sigma_\pi \circ i_{\pi(2\cdot j)}) 
      = 0
    ,
  \end{align*}
  where $\partial_n \colon C_n(X;\Z) \longrightarrow C_{n-1}(X;\Z)$ is
  the singular boundary operator. Hence, $\sigma_\pi \in C_*(M_\pi;\Z)$
  is a singular cycle. We write
  \[ 
     \alpha_\pi := [\sigma_\pi] \in H_n(M_\pi;\Z). 
  \]
\end{rem}

Every singular class that is represented by a single
singular simplex is a push-forward of one of these model classes:

\begin{lem}\label{lem:cyclematching}
  Let $X$ be a topological space, let $n
  \in \N_{>0}$, and let $\sigma \colon \Delta^n \longrightarrow X$ be a
  singular simplex.
  \begin{enumerate}
    \item If $R$ is a ring with unit and $\partial_n \sigma = 0$
      in~$C_*(X;R)$, then $n$ is odd.
    \item If $\partial_n \sigma =0$ in~$C_*(X;\Z)$, then 
      $n$ is odd and there exists an
      \mbox{$n$-matching}~$\pi$ and a continuous map~$f \colon
      M_\pi \longrightarrow X$ such that
      \[ H_n(f;\Z)(\alpha_\pi) = [\sigma] \in H_n(X;\Z). 
      \]
  \end{enumerate}
\end{lem}
\begin{proof}
  Let $R$ be a ring with unit and let 
  $0 = \partial_n \sigma = \sum_{j = 0}^n (-1)^j \cdot
  \sigma \circ i_j$. Because $\map(\Delta^{n-1}, X)$ is an $R$-basis
  of~$C_{n-1}(X;R)$, this sum has an even number of summands; in
  particular, $n$ is odd. 

  We now consider the case~$R = \Z$.  In order for these terms to cancel
  there has to be a bijection~$\pi \colon \{0,2,\dots,n-1\}
  \longrightarrow \{1,3,\dots,n\}$ such that
  \[ \sigma \circ i_j = \sigma \circ i_{\pi(j)}  
  \]
  holds for all~$j \in \{0,2,\dots,n-1\}$. Hence, $\sigma$ factors
  over the projection map~$\sigma_\pi \colon \Delta^n
  \longrightarrow\Delta^n/\sim_\pi = M_\pi$, i.e., there is a
  continuous map~$f \colon M_\pi \longrightarrow X$ with~$f \circ
  \sigma_\pi = \sigma$. By construction, $H_n(f;\Z)(\alpha_\pi) =
        [\sigma]$.
\end{proof}

\begin{rem}[cellular structure on model spaces]\label{rem:cw}
  Let $n \in \N$ be odd and let $\pi$ be an $n$-matching. 
  The face lattice of~$\Delta^n$ induces a CW-structure on the
  quotient~$M_\pi = \Delta^n / \sim_\pi$. 
\end{rem}

\begin{rem}[trivial extension of matchings]
  \label{rem:susp}
  Let $n \in \N$ be odd, let $\pi$ be an $n$-matching, and let 
  $\pi'$ be the $(n+2)$-matching with
  \[ \pi'|_{\{0,2,\dots,n-1\}} := \pi
     \qquad
     \text{and}
     \qquad
     \pi'(n+1) := n+2.
  \]
  Then $M_{\pi'} \cong_{\Top} \Sigma^2 M_{\pi}$, where $\Sigma$ denotes 
  the unreduced suspension. Indeed, this can be seen as follows: Let
  \begin{align*}
    \varphi \colon \Delta^n \times [-1,1]
    & \longrightarrow \Delta^{n+2}\\
    (x,t) & \longmapsto 
    \begin{cases}
      (1-t) \cdot x + t \cdot (\frac14 \cdot e_{n+1} 
                            + \frac34 \cdot e_{n+2})
     & \text{if $t \geq 0$}\\
      (1-|t|) \cdot x + |t| \cdot (\frac34 \cdot e_{n+1} 
                            + \frac14 \cdot e_{n+2})                              & \text{if $t < 0$}.
    \end{cases}
  \end{align*}
  Then it is not hard to show that the map
  \begin{align*}
    \Phi \colon \Delta^n \times [-1,1]^2 & \longrightarrow 
    M_{\pi'} = \Delta^{n+2} / \sim_{\pi'}
    \\
    (x,t,s) & \longmapsto 
    \begin{cases}
      (1-s) \cdot \varphi (x,t) + s \cdot e_{n+2} 
      & \text{if $t \geq 0$, $s\geq 0$}\\
      (1-|s|) \cdot \varphi (x,t) + |s| \cdot e_{n+1} 
      & \text{if $t \geq 0$, $s < 0$}\\
      (1-s) \cdot \varphi (x,t) + s \cdot e_{n+1} 
      & \text{if $t < 0$, $s\geq 0$}\\
      (1-|s|) \cdot \varphi (x,t) + |s| \cdot e_{n+2} 
      & \text{if $t < 0$, $s < 0$}
    \end{cases}
  \end{align*}
  induces a well-defined homeomorphism~$\Sigma^2 M_\pi \cong_{\Top} M_{\pi'}$ 
  (Figure~\ref{fig:suspension}).

  There is only one $1$-matching, and clearly the
  corresponding model space is homeomorphic to~$S^1$. Hence, by induction, 
  for all odd~$n \in \N$ the matching~$\pi$ given by
  \[ 0 \mapsto 1,\quad 2 \mapsto 3,\quad \dots,\quad n-1 \mapsto n
  \]
  leads to the model space~$M_\pi \cong_{\Top} S^n$. 

  However, not all model spaces are homotopy spheres 
  (Proposition~\ref{prop:modelmfd}).
\end{rem}

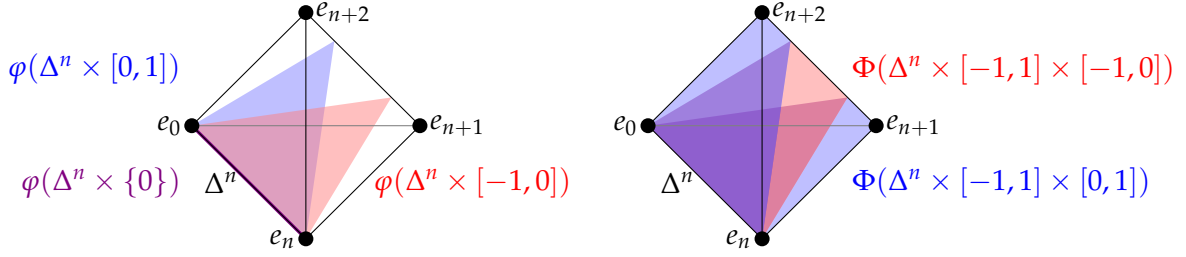
\begin{figure}
  \edef\ez{(0,0)}
  \edef\ezz{(1,-1)}
  \edef\eo{(2,0)}
  \edef\et{(1,1)}
  \def\gvertex#1{%
    \fill #1 circle (0.066);}
  \def\tetrahedron{%
      \draw \ez -- \ezz -- \eo;
      \draw[black!50] \ez -- \eo;
      \draw \eo -- \et -- \ez;
      \draw \et -- \ezz;
      \gvertex{\ez}
      \gvertex{\ezz}
      \gvertex{\eo}
      \gvertex{\et}
      \draw \ez node[anchor=east] {$e_0$};
      \draw \ezz node[anchor=east] {$e_n$};
      \draw (0.5,-0.5) node[anchor=east] {$\Delta^n$};
      \draw \eo node[anchor=west] {$e_{n+1}$};
      \draw \et node[anchor=west] {$e_{n+2}$};%
  }
  \def\ea{($0.25*(2,0)+0.75*(1,1)$)}
  \def\eb{($0.75*(2,0)+0.25*(1,1)$)}
  \begin{center}
    \makebox[0pt]{%
    \begin{tikzpicture}[x=1.5cm,y=1.5cm]
      \begin{scope}[opacity=0.5]
        \fill[blue!50] \ez -- \ezz -- \ea;
        \fill[red!50] \ez -- \ezz -- \eb;
      \end{scope}
      \draw[very thick, blue!50!red] \ez -- \ezz;
      \draw[blue!50!red] (0,-0.5) node[anchor=east] 
           {$\varphi(\Delta^n \times \{0\})$};
      \draw[blue] (0,0.5) node[anchor=east] 
           {$\varphi(\Delta^n \times [0,1])$};
      \draw[red] (1.5,-0.5) node[anchor=west] 
           {$\varphi(\Delta^n \times [-1,0])$};
      \tetrahedron
      \begin{scope}[shift={(4,0)}]
        \begin{scope}[opacity=0.5]
          \fill[blue!50] \ezz -- \eo -- \eb -- cycle;
          \fill[blue!50!red] \ezz -- \eb -- \ez;
          \fill[red!50] \ezz -- \eb -- \ea -- cycle;
          \fill[blue!50!red] \ezz -- \ea -- \ez -- cycle;
          \fill[blue!50] \ezz -- \ea -- \et -- \ez -- cycle;
        \end{scope}
        \draw[blue] (1.7,-0.5) node[anchor=west] 
             {$\Phi(\Delta^n \times [-1,1] \times [0,1])$};
        \draw[red] (1.7,0.5) node[anchor=west] 
             {$\Phi(\Delta^n \times [-1,1] \times [-1,0])$};
        \tetrahedron
      \end{scope}
    \end{tikzpicture}}
  \end{center}
  \caption{%
    The images of~$\varphi$ and~$\Phi$ from
    Remark~\ref{rem:susp}. We draw the images in~$\Delta^{n+2}$
    instead of in~$M_{\pi'} = \Delta^{n+2}/\sim_{\pi'}$.}
  \label{fig:suspension}
\end{figure}

\subsection{Fundamental group of model spaces}

We now show that these model spaces are simply connected:

\begin{prop}\label{prop:modelpi1}
  Let $n \in \N_{\geq 3}$ be odd, and let $\pi$ be an $n$-matching. Then 
  $M_\pi$ is simply connected.
\end{prop}

\begin{proof}
  Clearly, $M_\pi$ is connected. We now use the cellular structure
  from Remark~\ref{rem:cw} to compute the fundamental group of~$M_\pi$
  (with respect to the basepoint given by the $0$-vertex
  of~$\Delta^n$):

  As first step, we describe the \emph{$0$-skeleton} of~$M_\pi$: If $J, K \in
  \{0,\dots, n\}$ satisfying~$J < K$ are matched through~$\pi$, then the
  $(n-1)$-faces 
  \begin{align*}
    &[0,\dots, J-1,J+1, \dots , K-1, K, K+1, \dots, n] 
    \quad \text{and} \\
    &[0,\dots, J-1, J, \dots, K-2, K -1, K+1 \dots, n] 
  \end{align*}
  of~$\Delta^n$ are glued and so give the
  same $(n-1)$-cell in~$M_\pi = \Delta^n/\sim_\pi$. In particular, the
  vertices of~$\Delta^n$ numbered by~$J, J+1, \dots, K$ all project to
  the same $0$-cell of~$M_\pi$. Let $\sim$ denote the equivalence
  relation on~$\{0,\dots,n\}$ generated by
  \[ j \sim k 
     \Longleftrightarrow
     \exists_{J,K \in \{0,\dots, n\}} (J \leq j \leq k \leq K) \land 
                                \bigl( \pi(J) = K \lor \pi(K)=J \bigr).
  \]
  Then one can check easily for~$j, k \in \{0,\dots, n\}$ that the $j$-th
  and the $k$-vertex of~$\Delta^n$ lead to the same $0$-cell
  in~$M_\pi$ if and only if~$j \sim k$. 

  As next step, we describe a spanning tree for the \emph{$1$-skeleton} 
  of~$M_\pi$: Using the above description of the $0$-skeleton, it is
  not hard to see that the $1$-cells of~$M_\pi$ corresponding to the
  edges
  \[ T := 
      \bigl\{
      [j, j+1]
      \bigm| \text{$j \in \{0,\dots,n\}$ is maximal in its $\sim$-equivalence 
             class}
      \bigr\}
  \]
  of~$\Delta^n$ is the set of edges of a spanning tree for the
  $1$-skeleton of~$M_\pi$.

  The \emph{$2$-skeleton} of~$M_\pi$ now gives us the presentation 
  $\pi_1(M_\pi) 
     \cong \langle S \mid R \rangle$ 
  of the fundamental group, where 
  \begin{align*}  
    S := & \;\bigl\{ s_{[j,k]} \bigm| j,k \in \{0,\dots,n\},\ j < k\bigr\}\\
    R := & \;\{s_e \mid e \in T\}\\
         & \;\cup \{ s_{[k,\ell]} \cdot s_{[j,k]} = s_{[j,\ell]}
                \mid j,k,\ell \in \{0,\dots, n\},\ j < k < \ell\}\\
         & \;\cup \{ s_{[j,k]} = s_{[j',k']} 
                \mid j,k,j',k' \in \{0,\dots,n\},\ j<k, j'<k'
                   ,\ \sigma_\pi \circ [j,k] = \sigma_\pi \circ [j',k']
                \}.
  \end{align*}
  Using the first and second type of relations, we immediately see
  that $\pi_1(M_\pi)$ is generated by~$\{ s_{[j,j+1]} \mid j \in
  \{0,\dots, n-1\},\ [j,j+1] \not\in T \}$. So it suffices to show 
  that these elements are trivial in~$\pi_1(M_\pi)$. 
  
  Let $j \in \{0,\dots, n-1\}$ with $[j,j+1] \not\in T$. By construction 
  of~$T$, this means that there exist~$J,K \in \{0,\dots, n\}$ with
  \[ J \leq j < j+1 \leq K 
  \]
  and the property that $J$ and $K$ are matched through~$\pi$. 
  We now consider the following cases:
  \begin{itemize}
  \item Let $K \neq n$.  
    Because $J$ and $K$ are matched through~$\pi$, we have
    \[ \sigma_\pi \circ [j,n] = \sigma_\pi \circ [j+1, n],
    \]
    and so~$s_{[j,n]} = s_{[j+1,n]}$ in~$\pi_1(M_\pi)$. Then, 
    the $2$-face~$[j, j+1, n]$ of~$\Delta^n$ shows that
    \[ s_{[j,j+1]} = s^{-1}_{[j+1,n]} \cdot s_{[j,n]} = 1 
    \]
    in~$\pi_1(M_\pi)$.
  \item The case of $J \neq 0$ can be treated in the same way as the 
    previous case.
  \item Let $J = 0$ and $K = n$. Because $n \geq 3$ there is 
    at least one other matching between some~$J', K' \in \{0,\dots, n\}$ 
    with~$J' < K'$ through~$\pi$. The same argument as in the 
    first case yields~$s_{[J',J'+1]} = 1$ in~$\pi_1(M_\pi)$. On the 
    other hand, $\pi(0) = n$ gives 
    \[ s_{[0,1]} = s_{[1,2]} = \dots = s_{[n-1,n]}
    \]
    in~$\pi_1(M_\pi)$. In particular, $s_{[j,j+1]} = s_{[J',J'+1]} =
    1$ in~$\pi_1(M_\pi)$.
  \end{itemize}
  Hence, $\pi_1(M_\pi)$ is the trivial group.
\end{proof}

\subsection{How close are the model spaces to being manifolds?}

These model spaces are (after sufficiently many barycentric
subdivisions) $n$-dimensional pseudo-manifolds. However, in general,
they do not have the homotopy type of oriented closed connected
manifolds:

\begin{prop}\label{prop:modelmfd}
  \hfil
  \begin{enumerate}
  \item
    Let $n \in \{1,3\}$, and let $\pi$ be an $n$-matching. Then $M_\pi$ is 
    homotopy equivalent to~$S^n$.
  \item
    For every odd~$n \in \N_{\geq 5}$ there exists an
    $n$-matching~$\pi$ such that the model space~$M_\pi$ does
    \emph{not} have the homotopy type of an oriented closed connected
    $n$-manifold. In particular, $M_\pi \not\simeq S^n$.
  \end{enumerate}
\end{prop}

\begin{proof}
  \emph{Ad~1.} The matchings~$0 \mapsto 1$ and $0\mapsto1, 2 \mapsto
  3$ lead to~$S^1$ and $S^3$, respectively (Remark~\ref{rem:susp}). It
  remains to consider the matching~$\pi$ given by 
  \begin{align*} 
    0 & \longmapsto 3\\ 
    2 & \longmapsto 1. 
  \end{align*}
  Recall that any simply connected integral homology sphere is a
  homotopy sphere by the Hurewicz and Whitehead theorems. By
  Proposition~\ref{prop:modelpi1} it is thus sufficient to
  compute~$H_*(M_\pi;\Z)$ using the cellular structure from
  Remark~\ref{rem:cw}: A straightforward calculation shows that: 
  \begin{itemize}
    \item 
      By Proposition~\ref{prop:modelpi1}, we have $H_0(M_\pi;\Z) \cong
      \Z$ and $H_1(M_\pi;\Z) \cong 0$.
    \item The cellular chain complex of~$M_\pi$ in degree~$2$ has the $2$-cells 
      corresponding to the $2$-faces~$[0,1,2]$ and $[0,1,3]$ as a basis; in 
      degree~$1$, a basis is given by the $1$-faces~$[0,1], [0,3]$, and 
      with respect to these bases, the boundary map is represented by the 
      matrix
      \[ \begin{pmatrix}
           1 & 2\\
           0 & -1 \\
         \end{pmatrix}
      \]
      Thus, there are no cellular cycles in degree~$2$ and
      so~$H_2(M_\pi;\Z) \cong 0$.
    \item Because $\pi$ is a matching, the unique $3$-cell of~$M_\pi$
      is a cycle; as $M_\pi$ has no cells in dimension at least~$4$,
      it follows that $H_3(M_\pi;\Z) \cong \Z$ and $H_j(M_\pi;\Z)
      \cong 0$ for all~$j \in \N_{\geq 4}$.
  \end{itemize}
  Therefore, $H_*(M_\pi;\Z) \cong H_*(S^3;\Z)$.

  \emph{Ad~2.} We first prove the statement in dimension~$5$: To this end, 
  we consider the matching~$\pi$ given by
  \begin{align*}
    0 & \longmapsto 3\\
    2 & \longmapsto 5\\
    4 & \longmapsto 1.
  \end{align*}
  Using the cellular structure from Remark~\ref{rem:cw} one can
  calculate that
  \begin{align*}
    H_j(M_\pi;\Z) 
    \cong
    \begin{cases}
      \Z & \text{if $j = 0$} \\
      0  & \text{if $j=1$}\\
      0  & \text{if $j=2$}\\
      0  & \text{if $j=3$}\\
      \Z & \text{if $j = 4$, generated by~``$[1,2,3,4,5] - [0,1,3,4,5]$''}\\
      \Z & \text{if $j=5$, generated by~``$[0,\dots,5]$''}\\
      0  & \text{if $j \geq 6$}
    \end{cases}
  \end{align*}
  for all~$j \in \N$. 
  Hence, $H_*(M_\pi;\Z)$ is not compatible with
  Poincar\'e duality, and so $M_\pi$ cannot have the homotopy type of
  an oriented closed connected $5$-manifold.

  Let $k \in \N$.  Taking successive trivial extensions of this
  matching~$\pi$ as in Remark~\ref{rem:susp} leads to model spaces
  homeomorphic to~$\Sigma^{2 \cdot k} M_\pi$ of dimension~$5 + 2 \cdot
  k$. Because
  \[ H_j (\Sigma^{2 \cdot k} M_\pi;\Z) 
         \cong 
         \begin{cases}
           0 & \text{if $j \in \{1,\dots, 3 + 2 \cdot k\}$}\\
           \Z & \text{if~$j \in \{0, 4 + 2\cdot k, 5 +2 \cdot k\}$}
         \end{cases}
  \]
  is not compatible with Poincar\'e duality, these model spaces
  do not have the homotopy type of an oriented closed connected $(5 +
  2\cdot k)$-manifold.
\end{proof}

\section{Manifolds of integral simplicial volume equal to~$1$} 
\label{sec:1}

In this section, we will prove Theorem~\ref{thm:odd1} and its
generalisation to multiples of the fundamental class
(Theorem~\ref{thm:odd1mult}).

\subsection{Small fundamental class}

We follow the strategy outlined in the introduction. In particular, we
will use the model spaces introduced in Section~\ref{sec:model} and
the following refinement of the classical Betti number
bound~\cite[Example~14.28]{lueck}:

\begin{lem}\label{lem:oddhomology}
  Let $n \in \N$ be odd and let $M$ be an oriented closed connected
  \mbox{$n$-mani}\-fold.
  \begin{enumerate}
    \item Let $R$ be a ring with unit, let $m \in \Z$ be invertible 
      in~$R$, and suppose that $\ilone{m \cdot \fcl M} =1$. Then  
      for all~$k \in \{1,\dots, n-1\}$ we have
      \[ H_k(M;R) \cong 0. 
      \]
    \item
      In particular: If~$\isv M = 1$, then for all~$k \in
      \{1,\dots,n-1\}$ we have
      \[  H_k(M;\Z) \cong 0.
      \]
  \end{enumerate}
\end{lem}

If $X$ is a space and~$\alpha \in H_*(X;\Z)$, then 
\[ \ilone \alpha :=
   \min \bigl\{ |c|_1 \bigm| \text{$c \in C_*(X;\Z)$ is a cycle
                             representing~$\alpha$}\bigr\}.
\]  
For example, $\isv M = \ilone{\fcl M}$ holds for all oriented closed connected 
manifolds~$M$.

\begin{proof}
  It suffices to prove the first part. Because of~$\ilone{m \cdot \fcl
    M} = 1$ there is a singular simplex~$\sigma \colon \Delta^n
  \longrightarrow M$ that is a singular cycle and that satisfies~$[\sigma]
  = \pm m \cdot \fcl M \in H_n(M;\Z)$; without loss of generality, we 
  may assume~$[\sigma] = m \cdot \fcl M$. Because $m$ is a unit in~$R$,
  the Poincar\'e duality homomorphism
  \begin{align*}
    \varphi_k := 
    \args \cap m \cdot \fclr M
    \colon H^{n-k}(M;R) & \longrightarrow H_k(M;R)
    \\
    [f] & \longmapsto (-1)^{(n-k) \cdot k} \cdot 
    \bigl[ f(\bface {n-k} \sigma) \cdot \fface k \sigma \bigr] 
  \end{align*}
  with $R$-coefficients is an isomorphism for all~$k \in \{0,\dots,
  n\}$.  Hence, it is enough to show that $\varphi_k$ is the zero map
  for all~$k \in \{1, \dots, n-1\}$: If $k$ is even, then the
  $k$-simplex~$\fface k \sigma$ cannot be cycle
  (Lemma~\ref{lem:cyclematching}), and so the fact that $\varphi_k$ 
  as described above is well-defined shows that~$\varphi_k =
  0$.\footnote{Alternatively, one can use the arguments
for the odd case to show that $\bface{n-k}\sigma$ is a boundary, whence 
every cocycle vanishes on~$\bface{n-k}\sigma$.}
 
  Now let $k \in \{1,\dots,n-1\}$ be odd and let $\pi$ be an
  $n$-matching for~$\sigma$ as in Lemma~\ref{lem:cyclematching}.  We
  then define $I := \bigl\{ j\in \{0,2, \dots,k+1\} \bigm| \pi(j) \leq
  k \bigr\}$ and $\overline I := \{0,\dots,k+1\} \setminus (I \cup \pi(I))$.
  Moreover, for a sequence~$(j_0, \dots, j_r)$ in~$\{0,\dots, n\}$ we write
  \begin{align*}
    [j_0, \dots, j_r]_\sigma \colon \Delta^r & \longrightarrow M
    \\
    (t_0, \dots, t_r) & \longmapsto 
    \sigma \biggl(\sum_{s=0}^r (-1)^s \cdot t_s \cdot e_{j_s}\biggr)
  \end{align*}
  for the corresponding subsimplex of~$\sigma$, where, $e_0, \dots,
  e_n$ denote the standard unit vectors of~$\R^{n+1} \supset
  \Delta^n$. For example,~$\fface k \sigma = [0,\dots,k]_\sigma$. 
  With this notation and the relation $\sigma \circ i_j = \sigma
  \circ i_{\pi(j)}$ for all~$j \in \{0,2,\dots, n-1\}$, we
  obtain
  \begin{align*}
    \partial_{k+1} [0,\dots,k+1]_\sigma \hspace{-2cm}
    & 
    \\
    & = \sum_{j=0}^{k+1} (-1)^{j} \cdot [0,\dots,\widehat j, \dots, k+1]_\sigma\\
    & = \sum_{j \in I} [0,\dots,\widehat{\pi(j)}, \dots, k+1]_\sigma
      - \sum_{j \in \pi(I)} [0,\dots,\widehat j, \dots k+1]_\sigma
      + \sum_{j \in \overline I} (-1)^j \cdot [0,\dots, k]_\sigma
      \\
    & = 0 + \sum_{j \in \overline I} (-1)^j \cdot \fface k \sigma.
  \end{align*}
  Because $k$ is odd, also $|\overline I|$ is odd; moreover, $\pi$
  matches even with odd indices, so the number of even elements and
  the number of odd elements in~$\overline I$ differ by~$1$. Therefore,
  \[ \partial_{k+1} [0,\dots,k+1]_\sigma = \pm \fface k \sigma.
  \]
  Thus, $[\fface k \sigma] = 0$ in~$H_k(M;R)$, which
  implies~$\varphi_k = 0$.
\end{proof}


\begin{proof}[Proof of Theorem~\ref{thm:odd1}]
  Let $n \in \N$ be odd, and let $M$ be an oriented closed connected
  $n$-manifold with~$M \simeq S^n$. Then $\isv M = \isv
  {S^n}$. Moreover, $\isv{S^n} = 1$ for odd~$n$, as is witnessed by
  the cycle~$\Delta^n \longrightarrow \Delta^n/\partial \Delta^n
  \cong_{\Top} S^n$ given by the canonical projection.

  Conversely, let $n \in \N_{>0}$ and $M$ be an oriented closed connected
  $n$-manifold that satisfies~$\isv M = 1$. Then $n$ is odd by
  Lemma~\ref{lem:cyclematching}. Because $S^1$ is the only oriented
  closed connected \mbox{$1$-mani}\-fold we may assume~$n \geq 3$.  We
  follow the strategy outlined in the introduction:

  The manifold~$M$ is simply connected: In view of
  Lemma~\ref{lem:cyclematching} there is an $n$-matching~$\pi$ and a
  continuous map~$f \colon M_\pi \longrightarrow$ such that
  \[ H_n(f;\Z) (\alpha_\pi) = \fcl M \in H_n(M;\Z). 
  \]
  Because $M_\pi$ is simply connected, $f$ lifts to a map~$\widetilde
  f \colon M_\pi \longrightarrow \widetilde M$ to the universal
  covering~$p \colon \widetilde M \longrightarrow M$ of~$M$. Thus,
  \[    H_n(p;\Z) \circ H_n(\widetilde f;\Z)(\alpha_\pi)
  = H_n(f;\Z)(\alpha_\pi) 
  = \fcl M \neq 0, 
  \]
  and so~$H_n(\widetilde M;\Z) \not\cong 0$. In particular,
  $\widetilde M$ is compact and $p$ is a finite covering. Transfer
  shows that $\mathrm{im}\ H_n(p;\Z) = D \cdot \Z \cdot \fcl M$, where
  $D$ is the number of sheets of~$\pi$. Therefore, $D = 1$, and so $M
  \cong_{\Top} \widetilde M$ is simply connected.
  
  We have~$H_*(M;\Z) \cong H_*(S^n;\Z)$ by
  Lemma~\ref{lem:oddhomology}.
  
  Hence, $M \simeq S^n$: Any simply connected integral homology
  $n$-sphere is homotopy equivalent to~$S^n$ by the Hurewicz and Whitehead 
  theorems.
\end{proof}

\subsection{Small multiples of the fundamental class}

More generally, this technique also gives the following
characterisation for multiples of the fundamental class:

\begin{thm}[multiples of the fundamental class of small integral $\ell^1$-norm]
  \label{thm:odd1mult}
  Let $n \in \N$ be odd and let $M$ be an oriented closed connected
  $n$-manifold. Then the following are equivalent:
  \begin{enumerate}
    \item There exists~$m \in \Z\setminus\{0\}$ with~$\ilone{m \cdot
      \fcl M} = 1$.
    \item The manifold~$M$ is dominated by~$S^n$, i.e., there is a
      map~$S^n \longrightarrow M$ of non-zero degree.
    \item There is a subsequence of~$\bigl(\ilone{m \cdot \fcl
      M}\bigr)_{m \in \N}$ that is bounded by~$1$.
  \end{enumerate}
\end{thm}

\begin{proof}
  The implication~``3~$\Longrightarrow$~1'' is clear.  For the
  implication~``2~$\Longrightarrow$~3'' we use self-maps~$S^n
  \longrightarrow S^n$ of arbitrarily high degree, the fact that
  $\isv{S^n} = 1$ and that 
  \[ \ilone{H_*(f;\Z)(\alpha)} \leq \ilone \alpha\] 
  holds for all continuous maps~$f \colon X \longrightarrow Y$ and
  all~$\alpha \in H_*(X;\Z)$.

  We will now prove~``1~$\Longrightarrow$~2'': Without loss of
  generality we may assume that~$n > 1$. Let $m \in \Z \setminus
  \{0\}$ with~$\ilone{m \cdot \fcl M} = 1$. We will now proceed 
  as in the proof of Theorem~\ref{thm:odd1}:
  \begin{itemize}
    \item Reduction to the simply connected case: 
      Because~$\ilone{m
        \cdot \fcl M} = 1$, Lemma~\ref{lem:cyclematching} implies that
      there is an $n$-matching~$\pi$ and a continuous map~$f \colon M_\pi
      \longrightarrow M$ satisfying
      \[ H_n(f;\Z)(\alpha_\pi) = m \cdot \fcl M.
      \]
      The model space~$M_\pi$ is simply connected
      (Proposition~\ref{prop:modelpi1}); hence, $f$ lifts to a
      map~$\widetilde f \colon M_\pi \longrightarrow \widetilde M$ to
      the universal covering~$\widetilde M$ of~$M$. Because~$m \cdot
      \fcl M \neq 0$ it follows that $H_n(\widetilde M;\Z) \not\cong
      0$, and so $\widetilde M$ is compact and $H_n(\widetilde
      f;\Z)(\alpha_\pi)$ is a non-zero multiple of~$\fcl{\widetilde
        M}$ with integral $\ell^1$-norm equal to~$1$. 
    \item We have $H_*(\widetilde M;\Q) \cong H_*(S^n;\Q)$ by
      Lemma~\ref{lem:oddhomology} and the first step.
    \item Existence of a non-zero degree map~$g \colon S^n
      \longrightarrow \widetilde M$: Combining the previous steps with
      the rational Hurewicz theorem~\cite[Theorem~8.6]{fht} implies
      that the rational Hurewicz homomorphism
      \begin{align*}
        h_n^{\widetilde M} \colon 
        \pi_n(\widetilde M) \otimes_\Z \Q & \longrightarrow
        H_n(\widetilde M;\Z) \otimes_\Z \Q \cong H_n(\widetilde M;\Q)
        \\ [g]_* & \longmapsto H_n(g;\Z)(\fcl {S^n}) \otimes 1
      \end{align*}
      is surjective, and any (pointed) continuous map~$g
      \colon S^n \longrightarrow \widetilde M$ that is mapped
      via~$h_n^{\widetilde M}$ to an element of~$H_n(\widetilde M;\Q)
      \setminus \{0\}$ has non-zero degree.
  \end{itemize}
  Composing~$g \colon S^n \longrightarrow \widetilde M$ with the 
  universal covering map~$\widetilde M \longrightarrow M$ hence shows 
  that $M$ is dominated by~$S^n$.
\end{proof}

\subsection{Integral approximations of simplicial volume}\label{subsec:intapprox}

Determining the exact relation between ordinary simplicial volume and
$L^2$-Betti numbers of oriented closed connected aspherical manifolds
is a long-standing open
problem~\cite[p.~232]{gromovasymptotic}. Estimates between Betti
numbers and simplicial volumes can be obtained in the presence of some
integrality of coefficients~\cite{gromov,mschmidt}, based on
Poincar\'e duality and a counting argument as in 
Lemma~\ref{lem:oddhomology}. Hence, we are led to the question in
which sense ordinary simplicial volume can be approximated by integral
versions of simplicial volume and how these approximations relate to
Betti numbers and $L^2$-Betti numbers.

On the one hand, one can use the action of the fundamental group to
introduce more flexibility in the coefficients. This leads to stable
integral simplicial volume (based on the integral simplicial volumes of all
finite coverings) and integral foliated simplicial volume (based on
probability spaces with an action of the fundamental
group)~\cite{ffm,loehpagliantini}. In contrast to $L^2$-Betti numbers
that can be computed by the corresponding integral
approximations~\cite{lueckapprox,gaboriau}, stable integral simplicial
volume and integral foliated simplicial volume of oriented closed 
connected hyperbolic manifolds of dimension
at least~$4$ does \emph{not} coincide with ordinary simplicial
volume~\cite{ffm,flps}.

On the other hand, ordinary simplicial volume of an oriented closed
connected manifold~$M$ coincides with~\cite[Remark~5.4]{ffsn_rep}
\[ \inf_{d \in \N_{>0}} \frac1d \cdot \bigl\|{d \cdot \fcl M}\bigr\|_{1,\Z}.
\]
Therefore, one can use the sequence~$\Sigma(M) := (\| d\cdot \fcl M
\|_{1,\Z})_{d \in \N}$ to introduce refined versions of vanishing of
simplicial volume. The strongest such vanishing occurs in case the
sequence~$\Sigma(M)$ contains a subsequence bounded by~$1$. This is
exactly the situation of Theorem~\ref{thm:odd1mult}, which gives a
complete geometric characterisation of such manifolds. More generally,
the property that $\Sigma(M)$ contains a bounded subsequence is
related to weak flexibility of homology classes~\cite[Theorem~5.1,
  Remark~5.4]{ffsn_rep}.  However, it is not clear how~$\Sigma(M)$ is
related to $L^2$-Betti numbers in the aspherical case.

\section{Manifolds of integral simplicial volume equal to~$2$} 
\label{sec:2}

In the previous section, we gave a complete classification of
manifolds with integral simplicial volume equal to~$1$. The natural
next step is to look at manifolds with integral simplicial volume
equal to~$2$. 

\begin{question}
  Which oriented closed connected manifolds~$M$ satisfy~$\isv M = 2$\,?
\end{question}

However, this problem seems to be much harder than the case of
integral simplicial volume equal to~$1$. In the following, we will
give some partial answers.

We begin with a small, but helpful, observation, generalising the parity 
aspect of Lemma~\ref{lem:cyclematching}:

\begin{lem}[parity of cycles]\label{lem:parity}
  Let $M$ be an oriented closed connected manifold of even dimension. 
  Then $\isv M$ is even. 
\end{lem}
\begin{proof}
  Let $n := \dim M$ and let $c = \sum_{j=1}^k a_j \cdot \sigma_j \in
  C_n(M;\Z)$ be an integral fundamental cycle. Because $\Delta^n$ has
  an odd number of $(n-1)$-faces and because the singular
  $n$-simplices of~$M$ form a $\Z$-basis of~$C_n(M;\Z)$, the
  definition of the boundary operator yields that $\sum_{j=1}^k a_j$ is
  even. Then also $|c|_1 = \sum_{j=1}^k |a_j|$ is even and thus $\isv M$ 
  is even.
\end{proof}

\subsection{Dimension~$2$}

We start with the simple and well-known case of dimension~$2$: 

\def\gvertex(#1,#2){\fill (#1,#2) circle (0.1);} 
\def\ggvertex#1{\fill #1 circle (0.1);}

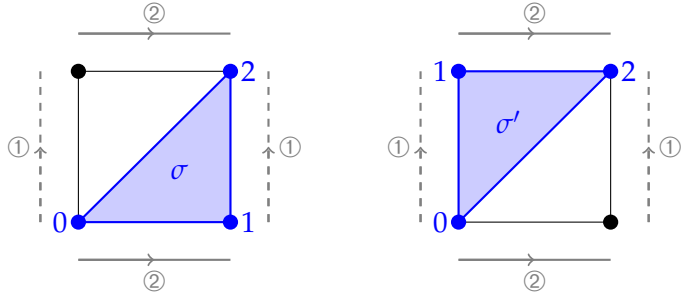
\begin{figure}
  \begin{center}
    \def\torusglue{%
      \begin{scope}[black!50]
      \draw[thick,dashed] (-0.5,0) -- (-0.5,2);
      \draw[thick,dashed] (2.5,0) -- (2.5,2);
      \draw[thick] (0,-0.5) -- (2,-0.5);
      \draw[thick] (0,2.5) -- (2,2.5);
       \draw[->,thick,dashed] (-0.5,0) -- (-0.5,1);
      \draw[->,thick,dashed] (2.5,0) -- (2.5,1);
      \draw[->,thick] (0,-0.5) -- (1,-0.5);
      \draw[->,thick] (0,2.5) -- (1,2.5);
      \draw (1,-0.5) node[anchor=north] {\ding{193}};
      \draw (1,2.5) node[anchor=south] {\ding{193}};
      \draw (-0.5,1) node[anchor=east] {\ding{192}};
      \draw (2.5,1) node[anchor=west] {\ding{192}};
      \end{scope}}
    \begin{tikzpicture}
      \torusglue
      \gvertex (0,2)
      \draw (0,0) -- (0,2) -- (2,2);
      \begin{scope}[blue,thick]
        \fill[blue!20] (0,0) -- (2,0) -- (2,2) -- cycle;
        \draw (0,0) -- (2,0) -- (2,2) -- cycle;
        \gvertex (0,0)
        \gvertex (2,0)
        \gvertex (2,2)
        \draw (0,0) node[anchor=east] {$0$};
        \draw (2,0) node[anchor=west] {$1$};
        \draw (2,2) node[anchor=west] {$2$};
        \draw (1.33,0.66) node {$\sigma$};
      \end{scope}
      \begin{scope}[shift={(5,0)}]
        \torusglue
        \gvertex (2,0)
        \draw (0,0) -- (2,0) -- (2,2);
        \begin{scope}[blue,thick]
          \fill[blue!20] (0,0) -- (0,2) -- (2,2) -- cycle;
          \draw (0,0) -- (0,2) -- (2,2) -- cycle;
          \gvertex (0,0)
          \gvertex (0,2)
          \gvertex (2,2)
          \draw (0,0) node[anchor=east] {$0$};
          \draw (0,2) node[anchor=east] {$1$};
          \draw (2,2) node[anchor=west] {$2$};
          \draw (0.66,1.33) node {$\sigma'$};
        \end{scope}
      \end{scope}
    \end{tikzpicture}
  \end{center}

  \caption{Fundamental cycle~$\sigma - \sigma'$ of the torus; the
    edges of the square are glued as indicated}
  \label{fig:toruscycle}
\end{figure}

\begin{prop}[integral simplicial volume of surfaces]\label{prop:surfaces}
  \hfil
  \begin{enumerate}
    \item We have~$\isv {S^2} = 2$.
    \item If $S$ is an oriented closed connected surface of genus~$g \in \N_{\geq 1}$, 
      then $\isv S = 4 \cdot g -2$.
  \end{enumerate}
\end{prop}
\begin{proof}
  In even dimension, there is no cycle consisting of a single simplex
  (Lemma~\ref{lem:cyclematching}). Thus, oriented closed
  connected surfaces~$S$ satisfy~$\isv S \geq 2$.
  Clearly, $\isv{S^2} = 2$ and
  $\isv{S^1 \times S^1} = 2$; the latter can be seen from the
  fundamental cycle depicted in Figure~\ref{fig:toruscycle}.

  Let $S$ be an oriented closed connected surface of genus~$g \in
  \N_{>1}$.  Then there is an explicit triangulation of the regular
  $4\cdot g$-gon that shows that $\isv{S} \leq 4 \cdot g -2$
  \cite{vbc,benedettipetronio}. On the other hand, we have
  \[ 4 \cdot g - 4  = \sv S \leq \isv S.
  \]
  No fundamental cycle on~$S$ realises the optimal value~$\sv S = 4
  \cdot g - 4$ because no straight singular simplex on~$\mathbb{H}^2$
  has maximal volume~\cite{jungreis}. Therefore, $\isv S > 4 \cdot g -
  4$. Because $\isv S$ is even (Lemma~\ref{lem:parity}), we obtain
  the desired estimate~$\isv S \geq 4 \cdot g -2$.
\end{proof}

\subsection{Dimension~$3$}

As a next step, we consider $3$-manifolds. 

\begin{prop}\label{prop:isv32}
  The following oriented closed connected $3$-manifolds~$M$ satisfy~$\isv M =2$:
  \[ S^1 \times S^2, 
     \quad
     \R P^3, 
     \quad
     L(3,1).
  \]
\end{prop}

\begin{proof} 
  In view of Theorem~\ref{thm:odd1}, all of these manifolds have
  integral simplicial volume at least~$2$.

  \emph{Ad~$S^1 \times S^2$:}
  By Proposition~\ref{prop:isvs1} below, $\isv {S^1 \times S^2} =
  2$. 

  \emph{Ad~$\R P^3$:} We describe an integral fundamental cycle
  for~$\R P^3$ that consists of two singular simplices: We view $\R
  P^3$ as quotient of the $3$-ball~$D^3$ modulo the antipodal action
  on~$\partial D^3 = S^2$. We consider the singular
  $3$-simplices~$\sigma, \sigma'$ on~$\R P^3$ as depicted in
  Figure~\ref{fig:rp3}. Then
  \[ \partial_3 (\sigma + \sigma') = 0 
  \]
  and looking at the local degree shows that $\sigma + \sigma'$
  represents~$\fcl{\R P^3}$.

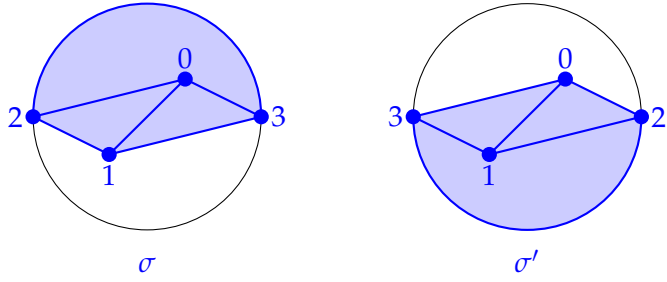
\begin{figure}
  \begin{center}
    \begin{tikzpicture}
      \draw (0,0) arc (180:360:1.5);
      \begin{scope}[color=blue]
        \fill[blue!20] (0,0) -- (1,-0.5) -- (3,0) -- 
        (3,0) arc (0:180:1.5) -- cycle;
        \draw (0,0) node[anchor=east] {$2$};
        \draw (1,-0.5) node[anchor=north] {$1$};
        \draw (2,0.5) node[anchor=south] {$0$};
        \draw (3,0) node[anchor=west] {$3$};
        \draw[thick] (3,0) arc (0:180:1.5);
        \draw[thick] (0,0) -- (1,-0.5) -- (3,0) -- (2,0.5) -- cycle;
        \draw[thick] (1,-0.5) -- (2,0.5);
        \gvertex(0,0); 
        \gvertex(3,0);
        \gvertex(1,-0.5);
        \gvertex(2,0.5);
        \draw (1.5,-2.2) node[anchor=south] {$\sigma$};
      \end{scope}
      \begin{scope}[shift={(5,0)}]
      \draw (3,0) arc (0:180:1.5);
      \begin{scope}[color=blue]
        \fill[blue!20] (3,0) -- (2,0.5) -- (0,0) -- 
        (0,0) arc (180:360:1.5) -- cycle;
        \draw (0,0) node[anchor=east] {$3$};
        \draw (1,-0.5) node[anchor=north] {$1$};
        \draw (2,0.5) node[anchor=south] {$0$};
        \draw (3,0) node[anchor=west] {$2$};
        \draw[thick] (0,0) arc (180:360:1.5);
        \draw[thick] (0,0) -- (1,-0.5) -- (3,0) -- (2,0.5) -- cycle;
        \draw[thick] (1,-0.5) -- (2,0.5);
        \gvertex(0,0); 
        \gvertex(3,0);
        \gvertex(1,-0.5);
        \gvertex(2,0.5);
        \draw (1.5,-2.2) node[anchor=south] {$\sigma'$};
      \end{scope}      
      \end{scope}
    \end{tikzpicture}
  \end{center}

  \caption{A fundamental cycle of~$\R P^3$ consisting of two singular
    simplices: The simplex~$\sigma$ fills the northern part of the
    ball with the depicted combinatorial structure; here, the
    face~$\sigma\circ i_0$ is the front half of the northern
    hemisphere, and $\sigma \circ i_3$ is the back half of the
    northern hemisphere. Similarly, the simplex~$\sigma'$ fills the
    southern part of the ball.}
  \label{fig:rp3}
\end{figure}

  \emph{Ad~$L(3,1)$:} The construction of a fundamental cycle of~$L(3,1)$ 
  consisting of two singular simplices is a little bit more delicate: 
  We view the lense space~$L(3,1)$ as quotient of
  the $3$-ball~$D^3$, where points on the southern 
  hemisphere of~$\partial D^3 = S^2$ are identified with points on the
  northern hemisphere after reflection at the equator and rotating
  around the North-South axis by~$2\cdot \pi/3$. We consider the
  two singular $3$-simplices~$\sigma, \sigma''$ on~$L(3,1)$ as depicted in
  Figure~\ref{fig:l31}. 

  By construction, $\partial_3(\sigma + \sigma'') = 0$ and looking at 
  the local degree (inside the ball) shows that $\sigma + \sigma''$ 
  represents~$2 \cdot \fcl{L(3,1)}$. We now manipulate the local degree 
  of~$\sigma''$ in order to obtain a fundamental cycle of~$L(3,1)$: We 
  may construct~$\sigma''$ in such a way that $\sigma''$ on the middle slice
  \[ S := \bigl\{ (t_0, t_1, t_2, t_3) \in \R^4 
          \bigm| t_0 + \dots + t_3  = 1,\ t_0 = t_3 \bigr\} 
  \]
  of~$\Delta^3$ is the obvious parametrisation of the equatorial disk
  in~$D^3$.  Let $P$ be the $\Delta^2$-shaped pillow obtained
  from~$\Delta^2 \times [0,1]$ by collapsing the vertical intervals
  in~$\partial \Delta^2 \times [0,1]$. Passing to the universal
  covering of~$L(3,1)$ one can easily construct a continuous map~$\tau
  \colon P \longrightarrow L(3,1)$ that coincides with~$\sigma''|_S$
  on the top and the bottom of~$P$ and that has local degree~$-3$. We
  then glue $\tau$ into~$\sigma''$ at~$S$; because gluing~$P$
  into~$\Delta^3$ at~$S$ yields a space that is homeomorphic
  to~$\Delta^3$ (while preserving the boundary~$\partial \Delta^3$),
  we thus obtain a singular simplex~$\sigma' \colon \Delta^3
  \longrightarrow L(3,1)$ of local degree~$1 - 3 = -2$
  with~$\partial_3 \sigma' = \partial_3 \sigma''$.

  Then $\partial_3(- \sigma - \sigma') = 0$ and looking at the local degree
  shows that the cycle~$-\sigma - \sigma'$ represents~$\fcl{L(3,1)}$.
\end{proof}

The lense space~$L(3,1)$ admits an obvious 
deomposition into two tetrahedra; however, it is not hard to see that
there is \emph{no} affine parametrisation of these tetrahedra that
yields an integral fundamental cycle of~$L(3,1)$.

\begin{figure}
  \begin{center}
    \begin{tikzpicture}
      \draw (1.5,0) circle (1.5); 
      \begin{scope}[color=blue,thick]
        \fill[blue!20] (1.5,0) circle (1.5);
        \draw (1.5,0) circle (1.5 and 0.5);
        \begin{scope}[shift={(1.5,0)}]
          \ggvertex{(0:1.5 and 0.5)};
          \ggvertex{(120:1.5 and 0.5)};
          \ggvertex{(240:1.5 and 0.5)};
          \ggvertex{(240:1.5 and 0.9)};
          \draw (240:1.5 and 0.5) -- (240:1.5 and 0.9);
          \draw (120:1.5 and 0.5) arc (120:240:1.45 and 0.7);
          \draw (240:1.5 and 0.9) arc (240:360:1.5 and 0.8);
          \draw (240:1.5 and 0.5) node[anchor=south] {$3$};
          \draw (240:1.5 and 0.9) node[anchor=north] {$0$};
          \draw (120:1.5 and 0.5) node[anchor=south] {$1$};
          \draw (0:1.5 and 0.5) node[anchor=west] {$2$};
        \end{scope}
        \draw (1.5,-2.2) node[anchor=south] {$\sigma$};
      \end{scope}
      \begin{scope}[shift={(5,0)}]
      \draw (1.5,0) circle (1.5); 
      \begin{scope}[color=blue,thick]
        \fill[blue!20] (1.5,0) circle (1.5);
        \draw (1.5,0) circle (1.5 and 0.5);
        \begin{scope}[shift={(1.5,0)}]
          \ggvertex{(0:1.5 and 0.5)};
          \ggvertex{(120:1.5 and 0.5)};
          \ggvertex{(240:1.5 and 0.5)};
          \ggvertex{(240:1.5 and 0.9)};
          \draw (240:1.5 and 0.5) -- (240:1.5 and 0.9);
          \draw (120:1.5 and 0.5) arc (120:240:1.45 and 0.7);
          \draw (240:1.5 and 0.9) arc (240:360:1.5 and 0.8);
          \draw (240:1.5 and 0.5) node[anchor=south] {$0$};
          \draw (240:1.5 and 0.9) node[anchor=north] {$3$};
          \draw (120:1.5 and 0.5) node[anchor=south] {$2$};
          \draw (0:1.5 and 0.5) node[anchor=west] {$1$};
        \end{scope}
        \draw (1.5,-2.2) node[anchor=south] {$\sigma''$};
      \end{scope}
      \end{scope}
    \end{tikzpicture}
  \end{center}

  \caption{A fundamental cycle of~$L(3,1)$ consisting of two singular
    simplices: The simplex~$\sigma$ fills the ball, where $\sigma
    \circ i_0$ covers the northern hemisphere, $\sigma \circ i_3$
    covers the southern hemisphere and the faces~$\sigma \circ i_1$,
    $\sigma \circ i_2$ lie on the equator. The simplex~$\sigma''$ is 
    constructed in a similar way; here, $\sigma'' \circ i_0$ covers 
    the southern hemisphere and $\sigma'' \circ i_3$ covers the 
    northern hemisphere. The simplex~$\sigma'$ is then obtained from~$\sigma''$ 
    by inserting a pillow of local degree~$-3$. 
  }
  \label{fig:l31}
\end{figure}
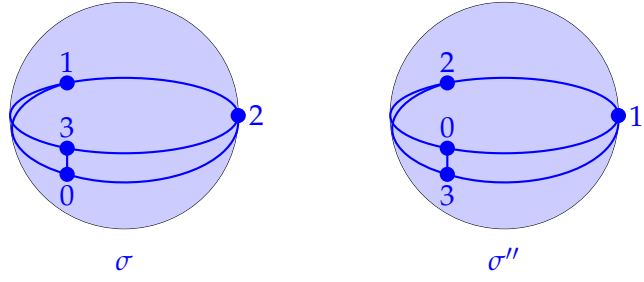

\begin{rem}\label{rem:32}
  If $M$ is an oriented closed connected $3$-manifold that is
  \emph{not} homotopy equivalent/homeomorphic to~$S^1 \times S^2$, $\R P^3$
  or~$L(3,1)$, then we have $\isv M \neq 2$: The model space construction of
  Section~\ref{sec:model} can be extended to the case of singular
  cycles that consist of more than one simplex.  If $M$ has a
  fundamental cycle~$c$ consisting of two singular simplices~$\sigma,
  \sigma' \colon \Delta^3 \longrightarrow M$, then (depending on the
  signs of the coefficients) there are two types of matchings to be
  considered:
  \begin{itemize}
    \item If $c = \sigma + \sigma'$ (or~$c = - \sigma - \sigma')$,
      then one obtains a bijection
      \[ \{0, 2, 0', 2' \} \longrightarrow \{1,3,1',3'\},
      \]
      where the sets $\{0, \dots, n\}$ and $\{0', \dots, n'\}$ 
      correspond to the faces of~$\sigma$ and $\sigma'$ respectively.
    \item If $c = \sigma - \sigma'$, then one obtains a 
      bijection
      \[ \{0, 2, 1', 3'\} \longrightarrow \{1,3,0'2'\}. 
      \]
  \end{itemize}

  Similarly to the proof of Proposition~\ref{prop:modelpi1}, a lengthy
  calculation shows that the (connected components of the)
  corresponding model spaces constructed out of two copies
  of~$\Delta^3$ whose facets are glued according to these matchings
  have cyclic fundamental group. 
  Moreover, we obtain (e.g., by computer calculations) the integral
  homology groups of these model spaces in
  Table~\ref{table:computation32plus} and
  Table~\ref{table:computation32minus}.

  \begin{table}
    {
    \begin{tabular}{llllllll}
      \makebox[0pt][l]{Matching}
      & & & &\makebox[0pt][l]{$H_*(\args;\Z)$ in degree}
      \\
      & & & & 0 & 1 & 2 & 3
      \\[.2em]
      \hline
      \\[.5em]
      $(0, 1)$ & $(2, 3)$ & $(0', 1')$ & $(2', 3')$ & $\Z^2$ & $0$ & $0$ & $\Z^2$\\
$(0, 3)$ & $(2, 1)$ & $(0', 1')$ & $(2', 3')$ & $\Z^2$ & $0$ & $0$ & $\Z^2$\\
$(0, 1')$ & $(2, 3)$ & $(0', 1)$ & $(2', 3')$ & $\Z$ & $0$ & $0$ & $\Z$\\
$(0, 3)$ & $(2, 1')$ & $(0', 1)$ & $(2', 3')$ & $\Z$ & $0$ & $\Z$ & $\Z$\\
$(0, 1')$ & $(2, 1)$ & $(0', 3)$ & $(2', 3')$ & $\Z$ & $0$ & $0$ & $\Z$\\
$(0, 1)$ & $(2, 1')$ & $(0', 3)$ & $(2', 3')$ & $\Z$ & $0$ & $0$ & $\Z$\\
$(0, 3')$ & $(2, 1')$ & $(0', 3)$ & $(2', 1)$ & $\Z$ & $0$ & $\Z$ & $\Z$\\
$(0, 1')$ & $(2, 3')$ & $(0', 3)$ & $(2', 1)$ & $\Z$ & $0$ & $\Z^2$ & $\Z$\\
$(0, 1')$ & $(2, 3)$ & $(0', 3')$ & $(2', 1)$ & $\Z$ & $0$ & $\Z$ & $\Z$\\
$(0, 3')$ & $(2, 3)$ & $(0', 1')$ & $(2', 1)$ & $\Z$ & $0$ & $0$ & $\Z$\\
$(0, 3)$ & $(2, 3')$ & $(0', 1')$ & $(2', 1)$ & $\Z$ & $0$ & $\Z$ & $\Z$\\
$(0, 3)$ & $(2, 1')$ & $(0', 3')$ & $(2', 1)$ & $\Z$ & $\Z/3$ & $0$ & $\Z$\\
$(0, 3')$ & $(2, 1)$ & $(0', 3)$ & $(2', 1')$ & $\Z$ & $0$ & $0$ & $\Z$\\
$(0, 1)$ & $(2, 3')$ & $(0', 3)$ & $(2', 1')$ & $\Z$ & $0$ & $0$ & $\Z$\\
$(0, 1)$ & $(2, 3)$ & $(0', 3')$ & $(2', 1')$ & $\Z^2$ & $0$ & $0$ & $\Z^2$\\
$(0, 3')$ & $(2, 3)$ & $(0', 1)$ & $(2', 1')$ & $\Z$ & $0$ & $0$ & $\Z$\\
$(0, 3)$ & $(2, 3')$ & $(0', 1)$ & $(2', 1')$ & $\Z$ & $\Z/2$ & $0$ & $\Z$\\
$(0, 3)$ & $(2, 1)$ & $(0', 3')$ & $(2', 1')$ & $\Z^2$ & $0$ & $0$ & $\Z^2$\\
$(0, 3')$ & $(2, 1)$ & $(0', 1')$ & $(2', 3)$ & $\Z$ & $0$ & $0$ & $\Z$\\
$(0, 1)$ & $(2, 3')$ & $(0', 1')$ & $(2', 3)$ & $\Z$ & $0$ & $0$ & $\Z$\\
$(0, 1)$ & $(2, 1')$ & $(0', 3')$ & $(2', 3)$ & $\Z$ & $0$ & $\Z$ & $\Z$\\
$(0, 3')$ & $(2, 1')$ & $(0', 1)$ & $(2', 3)$ & $\Z$ & $0$ & $\Z^2$ & $\Z$\\
$(0, 1')$ & $(2, 3')$ & $(0', 1)$ & $(2', 3)$ & $\Z$ & $\Z/2$ & $0$ & $\Z$\\
$(0, 1')$ & $(2, 1)$ & $(0', 3')$ & $(2', 3)$ & $\Z$ & $\Z/2$ & $0$ & $\Z$
    \end{tabular}}

    \bigskip

    \caption{Integral homology of model spaces for cycles consisting of two singular $3$-simplices with the same sign}
    \label{table:computation32plus}
  \end{table}

  \begin{table}
    {
    \begin{tabular}{llllllll}
      \makebox[0pt][l]{Matching}
      & & & &\makebox[0pt][l]{$H_*(\args;\Z)$ in degree}
      \\
      & & & & 0 & 1 & 2 & 3
      \\[.2em]
      \hline
      \\[.5em]
      $(0, 1)$ & $(2, 3)$ & $(1', 0')$ & $(3', 2')$ & $\Z^2$ & $0$ & $0$ & $\Z^2$\\
$(0, 3)$ & $(2, 1)$ & $(1', 0')$ & $(3', 2')$ & $\Z^2$ & $0$ & $0$ & $\Z^2$\\
$(0, 0')$ & $(2, 3)$ & $(1', 1)$ & $(3', 2')$ & $\Z$ & $0$ & $0$ & $\Z$\\
$(0, 3)$ & $(2, 0')$ & $(1', 1)$ & $(3', 2')$ & $\Z$ & $0$ & $\Z$ & $\Z$\\
$(0, 0')$ & $(2, 1)$ & $(1', 3)$ & $(3', 2')$ & $\Z$ & $0$ & $0$ & $\Z$\\
$(0, 1)$ & $(2, 0')$ & $(1', 3)$ & $(3', 2')$ & $\Z$ & $0$ & $0$ & $\Z$\\
$(0, 2')$ & $(2, 0')$ & $(1', 3)$ & $(3', 1)$ & $\Z$ & $0$ & $\Z$ & $\Z$\\
$(0, 0')$ & $(2, 2')$ & $(1', 3)$ & $(3', 1)$ & $\Z$ & $0$ & $\Z$ & $\Z$\\
$(0, 0')$ & $(2, 3)$ & $(1', 2')$ & $(3', 1)$ & $\Z$ & $0$ & $0$ & $\Z$\\
$(0, 2')$ & $(2, 3)$ & $(1', 0')$ & $(3', 1)$ & $\Z$ & $0$ & $0$ & $\Z$\\
$(0, 3)$ & $(2, 2')$ & $(1', 0')$ & $(3', 1)$ & $\Z$ & $0$ & $\Z$ & $\Z$\\
$(0, 3)$ & $(2, 0')$ & $(1', 2')$ & $(3', 1)$ & $\Z$ & $0$ & $0$ & $\Z$\\
$(0, 2')$ & $(2, 1)$ & $(1', 3)$ & $(3', 0')$ & $\Z$ & $0$ & $0$ & $\Z$\\
$(0, 1)$ & $(2, 2')$ & $(1', 3)$ & $(3', 0')$ & $\Z$ & $0$ & $\Z$ & $\Z$\\
$(0, 1)$ & $(2, 3)$ & $(1', 2')$ & $(3', 0')$ & $\Z^2$ & $0$ & $0$ & $\Z^2$\\
$(0, 2')$ & $(2, 3)$ & $(1', 1)$ & $(3', 0')$ & $\Z$ & $0$ & $\Z$ & $\Z$\\
$(0, 3)$ & $(2, 2')$ & $(1', 1)$ & $(3', 0')$ & $\Z$ & $\Z$ & $\Z$ & $\Z$\\
$(0, 3)$ & $(2, 1)$ & $(1', 2')$ & $(3', 0')$ & $\Z^2$ & $0$ & $0$ & $\Z^2$\\
$(0, 2')$ & $(2, 1)$ & $(1', 0')$ & $(3', 3)$ & $\Z$ & $0$ & $0$ & $\Z$\\
$(0, 1)$ & $(2, 2')$ & $(1', 0')$ & $(3', 3)$ & $\Z$ & $0$ & $0$ & $\Z$\\
$(0, 1)$ & $(2, 0')$ & $(1', 2')$ & $(3', 3)$ & $\Z$ & $0$ & $0$ & $\Z$\\
$(0, 2')$ & $(2, 0')$ & $(1', 1)$ & $(3', 3)$ & $\Z$ & $0$ & $\Z$ & $\Z$\\
$(0, 0')$ & $(2, 2')$ & $(1', 1)$ & $(3', 3)$ & $\Z$ & $0$ & $0$ & $\Z$\\
$(0, 0')$ & $(2, 1)$ & $(1', 2')$ & $(3', 3)$ & $\Z$ & $0$ & $0$ & $\Z$
    \end{tabular}}

    \bigskip

    \caption{Integral homology of model spaces for cycles consisting of two singular $3$-simplices with opposite signs}
    \label{table:computation32minus}
  \end{table}

  \forget{%
  \begin{table}
    {\tiny
    \begin{tabular}{llllllll}
      \makebox[0pt][l]{Matching}
      & & & &\makebox[0pt][l]{$H_*(\args;\Z)$ in degree}
      \\
      & & & & 0 & 1 & 2 & 3
      \\[.2em]
      \hline
      \\[.5em]
      \input{computations32.txt}
    \end{tabular}}

    \bigskip

    \caption{Integral homology of model spaces for cycles consisting of two singular $3$-simplices}
    \label{table:computation32}
  \end{table}}

  Hence, all these model spaces have first integral homology group
  isomorphic to 
  \[ 0, \quad \Z/2, \quad \Z/3, \quad\text{or}\quad \Z.
  \]
  Because these model spaces have cyclic (in particular, Abelian) fundamental
  group, these are also the only possible isomorphism types of
  fundamental groups of such model spaces.

  Let $\pi$ be the matching corresponding to the considered fundamental 
  cycle~$c$ of~$M$, let $N_\pi = \Delta^3 \sqcup \Delta^3 /\!\!\sim$ be the 
  associated model space, let $\beta_\pi \in H_3(N_\pi;\Z)$ be the 
  model class and let $f \colon N_\pi \longrightarrow M$ be the continuous 
  map modelling~$c$. Then $H_3(f;\Z)(\beta_\pi) = \fcl M$, and 
  \begin{align*}
    \xymatrix{%
      H_k(N_\pi;\Z) 
      \ar[r]^-{H_k(f;\Z)}
      & H_k(M;\Z)
      \\
      H^{3-k}(N_\pi;\Z)
      \ar[u]^{\args\cap \beta_\pi}
      & H^{3-k}(M;\Z)
      \ar[l]^-{H^{3-k}(f;\Z)}
      \ar[u]^{\cong}_{\args \cap \fcl M}
    }
  \end{align*}
  is commutative for all~$k \in \N$ in view of the naturality of the
  cap-product. The right vertical arrow is an isomorphism by 
  Poincar\'e duality. Hence, $H_k(f;\Z) \colon H_k(N_\pi;\Z) \longrightarrow 
  H_k(M;\Z)$ is split surjective.

  Moreover, covering theory shows that $\pi_1(f) \colon \pi_1(N_\pi)
  \longrightarrow \pi_1(M)$ is surjective. Because $\pi_1(N_\pi)$ is
  cyclic, also $\pi_1(M)$ is cyclic and thus isomorphic
  to~$H_1(M;\Z)$. As we have seen above, $H_1(f;\Z)$ is split
  surjective, and so $\pi_1(f) \colon \pi_1(N_\pi;\Z) \longrightarrow
  \pi_1(M)$ is \emph{split} surjective. Therefore, the fundamental 
  group of~$M$ is isomorphic to
  \[ 0, \quad \Z/2, \quad \Z/3, \quad\text{or}\quad \Z.
  \]
  Now the classification of oriented closed connected
  $3$-manifolds~\cite{afw}  shows that $M$ is homotopy equivalent
  (even homeomorphic) to~$S^3$, $\R P^3$, $L(3,1)$, or $S^1 \times S^2$.  
  As we know that $\isv{S^3} = 1 \neq 2$ the only remaining candidates 
  for~$M$ are $\R P^3$, $L(3,1)$, or $S^1 \times S^2$, as claimed.
\end{rem}

\subsection{Higher dimensions}

The example of the torus generalises as follows:

\begin{prop}\label{prop:isvs1}
  For all~$n \in \N_{>1}$ we have
  \[ \isv{S^1 \times S^{n-1}} = 2
     =
     \begin{cases}
       \ \isv{S^1} \cdot \isv{S^{n-1}} & \text{if $n$ is odd}\\
       \ 2 \cdot \isv{S^1} \cdot \isv{S^{n-1}} & \text{if $n$ is even}.
     \end{cases}
  \]
\end{prop}
\begin{proof}
  Theorem~\ref{thm:odd1} shows that $\isv{S^1 \times S^{n-1}} > 1$
  because $S^1 \times S^{n-1} \not\simeq S^n$.  Therefore, it suffices
  to construct a fundamental cycle~$\sigma - \sigma'$ of~$S^1 \times
  S^{n-1}$ consisting of two singular simplices~$\sigma, \sigma'
  \colon \Delta^n \longrightarrow S^1 \times S^{n-1}$. We now
  describe the construction of such singular simplices; to this end we
  construct the two factors of these maps separately.

  The $S^1$-factors are defined on~$\Delta^n \subset \R^{n+1}$ by 
  \begin{align*}
    \sigma_1 = \sigma'_1 \colon \Delta^n & \longrightarrow S^1 \subset \C \\
    (t_0, \dots, t_n) & \longmapsto e^{2 \pi i \cdot \sum_{j=0}^n j \cdot t_j}.
  \end{align*}

  In order to understand the definition of the $S^{n-1}$-factors we first 
  give the corresponding matchings:
  \begin{itemize}
    \item If $n$ is even, then we consider the matching~$\pi$ given by
      \begin{align*} 
        0 & \leftrightarrow n' \\
        1 & \leftrightarrow 1' \\
        2 & \leftrightarrow 2' \\
        & \ \ \  \vdots\\
        n-1 & \leftrightarrow (n-1)' \\
        n & \leftrightarrow 0'.
      \end{align*}
      Moreover, we define
      \begin{align*} 
        d_0 & := 1 =: d'_n &
        d_1 & := 1 =: d'_1 \\
        \forall_{j \in \{2, \dots, n-1} \quad d_j & := 0 =: d'_j &
        d_n & := 0 =: d'_0.
      \end{align*}
    \item If $n$ is odd, then we consider the matching~$\pi$ given by
      \begin{align*}
        0 & \leftrightarrow n \\
        1 & \leftrightarrow 1' \\
        2 & \leftrightarrow 2' \\
        & \ \ \ \vdots \\
        n-1 & \leftrightarrow (n-1)' \\
        0' & \leftrightarrow n'.
      \end{align*}
      Moreover, we define
      \begin{align*} 
        d_0 & := 1 =: d_n &
        d_1 & := 1 =: d'_1 \\
        \forall_{j \in \{2, \dots, n-1} \quad d_j & := 0 =: d'_j &
        d'_0 & := 0 =: d'_n.
      \end{align*}
  \end{itemize}

  We choose $x_0 \in S^{n-1}$ and a homeomorphism~$\Delta^{n-1} /
  \partial \Delta^{n-1}$ that maps the point associated 
  with~$\partial \Delta^{n-1}$ to~$x_0$. We denote the constant map with
  value~$x_0$ by~$c \colon \Delta^{n-1} \longrightarrow S^{n-1}$ and
  the projection corresponding to the above homeomorphism by~$p \colon
  \Delta^{n-1} \longrightarrow \Delta^{n-1} / \partial \Delta^{n-1}
  \cong_{\Top} S^{n-1}$. 
  For~$j \in \{0,\dots, n\}$ we then define
  \[ \sigma_{2,j} := 
     \begin{cases}
       c & \text{if $d_j =0$}
       \\
       p & \text{if~$d_j =1$.}
     \end{cases}
  \]
  Because $\sum_{j = 0}^n (-1)^j \cdot d_j = 0$, these maps can be extended 
  to a continuous map~$\sigma_2 \colon \Delta^n \longrightarrow S^{n-1}$ 
  with~$\sigma_2 \circ i_j = \sigma_{2,j}$ for all~$j \in \{0,\dots, n\}$. 
  In the same way, we construct~$\sigma'_2 \colon \Delta^n
  \longrightarrow S^{n-1}$ using~$d'_0, \dots, d'_n$. 

  Finally, we set
  \[ \sigma := (\sigma_1, \sigma_2)
     \quad\text{and}\quad 
     \sigma' := (\sigma'_1, \sigma'_2).\]
  Because the $d_0,\dots, d_n$ and $d'_0, \dots, d'_n$
  are compatible with the matching~$\pi$, the faces of~$\sigma_2$
  and~$\sigma'_2$ cancel according to~$\pi$. Also $\sigma_1$ and
  $\sigma'_1$ are compatible with the matching~$\pi$ because for
  all~$(t_0, \dots, t_n) \in \Delta^n$ we have~$\sum_{j=0}^n t_j = 1$, and 
  so $\sigma_1 \circ i_0 = \sigma_1 \circ i_n$ etc.
  Thus, $\sigma - \sigma'$ is a singular cycle on~$S^1 \times S^{n-1}$.

  It remains to show that $\sigma - \sigma'$ represents~$\fcl{S^1
    \times S^{n-1}}$. To this end we choose singular cocycles~$f_1
  \in~C^1(S^1;\Z)$ and $f_2 \in C^{n-1}(S^{n-1};\Z)$ dual to~$\fcl{S^1}$
  and $\fcl{S^{n-1}}$ respectively; without loss of generality, we 
  may assume that $f_2(c) = 0$. 
  Then the cohomological cross-product~$f_1 \times f_2$ is a cocycle 
  on~$S^1 \times S^{n-1}$ that is dual to~$\fcl{S^1} \times \fcl{S^{n-1}} 
  = \fcl{S^1 \times S^{n-1}}$. Therefore, it suffices to show that
  $(f_1 \times f_2) (\sigma - \sigma') = 1$. Using the explicit 
  description of the cohomological cross-product via the Alexander-Whitney 
  map, we obtain
  \begin{align*}
    (f_1 \times f_2) (\sigma - \sigma') 
    & =
    (-1)^{n-1} \cdot 
   \bigl(f_1 (\fface 1 {\sigma_1}) \cdot f_2(\bface {n-1} {\sigma_2})
  - f_1 (\fface 1 {\sigma'_1}) \cdot f_2(\bface {n-1} {\sigma'_2})
    \bigr)
    \\
    & = (-1)^{n-1} \cdot (1 \cdot d_n - 1 \cdot d'_n)
    \\
    & = 1.
  \end{align*}
  Therefore, $\sigma - \sigma'$ is a fundamental cycle of~$S^1 \times S^{n-1}$.
\end{proof}

In particular, in even dimensions spheres are not the only oriented 
closed connected manifolds with integral simplicial volume equal to~$2$.

We now turn our attention to products of higher-dimensional spheres. In certain 
cases, we can show that these have integral simplicial volume larger than~$2$. 
However, the general case is still wide open.

\begin{lem}\label{lem:betti2}
  Let $M$ be an oriented closed connected $n$-manifold with~$\isv M =2$. 
  Then for all~$k \in \N$ we have:
  \begin{enumerate}
  \item If $k$ is even, then $H_k(M;\Z)$ is cyclic.
  \item If $k$ is odd, then $H_k(M;\Z)$ can be generated by two elements.
  \end{enumerate}
\end{lem}
\begin{proof}
  Because of $\isv M = 2$, there exist singular simplices~$\sigma,
  \tau \colon \Delta^n \longrightarrow M$ such that $\sigma + \tau$
  or~$\sigma - \tau$ is a cycle in~$C_n(M;\Z)$ representing~$\pm
  \fcl M$. Without loss of generality, we may assume that $[\sigma +
    \varepsilon \cdot \tau] = \fcl M \in H_n(M;\Z)$ where~$\varepsilon
  \in \{-1,1\}$.

  It suffices to treat the case $k \in \{1,\dots, n-1\}$. As in the
  proof of Lemma~\ref{lem:oddhomology}, we consider the Poincar\'e 
  duality isomorphism
  \begin{align*}
    \varphi_k := 
    \args \cap \fcl M
    \colon H^{n-k}(M;\Z) & \longrightarrow H_k(M;\Z)
    \\
    [f] & \longmapsto (-1)^{(n-k) \cdot k} \cdot 
    \bigl[ f(\bface {n-k} \sigma) \cdot \fface k \sigma\\
      & \phantom{\longmapsto (-1)^{(n-k) \cdot k}\quad}
      + \varepsilon \cdot f(\bface {n-k} \tau) \cdot \fface k \tau\bigr] 
  \end{align*}
  Clearly, the image of~$\varphi_k$ can be generated by two elements
  or less. If $k$ is even, then neither $\fface k \sigma$ nor $\fface
  k \tau$ is a cycle (Lemma~\ref{lem:cyclematching}); hence, there
  is (up to integral rescaling) at most one $\Z$-linear combination of
  $\fface k \sigma$ and $\fface k \tau$ that is a cycle. Thus, the
  image of~$\varphi_k$ is cyclic. Because $\varphi_k$ is an
  isomorphism and hence surjective, the result follows.
\end{proof}

In general, these bounds cannot be improved as the example of the 
torus in dimension~$2$ shows (Proposition~\ref{prop:surfaces}).

\begin{prop}
  Let $n \in \N$ be even. Then 
  \[ \isv{S^n \times S^n} \geq 4 = \isv{S^n} \cdot \isv{S^n}.
  \]
\end{prop}
\begin{proof}
  On the one hand, $\isv{S^n \times S^n}$ is even
  (Lemma~\ref{lem:parity}). On the other hand, by
  Lemma~\ref{lem:betti2}, we have $\isv {S^n \times S^n} > 2$
  because $H_n(S^n \times S^n;\Z) \cong \Z^2$ and $n$ is assumed to be
  even.
\end{proof}

\begin{rem}\label{rem:s2s4}
  We have
  \[ \isv{S^2 \times S^4} \geq 4 = \isv{S^2} \cdot \isv{S^4}. 
  \]
  Indeed, $\isv{S^2 \times S^4}$ is even by
  Lemma~\ref{lem:parity}. Because this integral simplicial volume is
  non-zero it suffices to prove that $\isv{S^2 \times S^4} \neq 2$.
  
  As in Remark~\ref{rem:32} we consider model spaces for cycles 
  consisting of two singular $6$-simplices. All such cycles are 
  modelled by matchings that are bijections
  \[ \{0,2,4,6, 1', 3',5' \} \longrightarrow \{1,3,5, 0',2',4',6'\}. 
  \]
  Calculating the homology of the model spaces of all possible
  matchings (via a computer) shows that they all have second integral
  homology that does not surject onto~$\Z$. On the other hand, if
  $\isv{S^2 \times S^4}$ were equal to~$2$, then the cap-product
  argument of Remark~\ref{rem:32} shows that the model map on second
  integral homology would need to be surjective. This shows that
  $\isv{S^2 \times S^4} \neq 2$.
\end{rem}

\begin{question}[integral simplicial volume of products]
  If $M$ and $N$ are oriented closed connected manifolds, then the
  ordinary simplicial volume (with \mbox{$\R$-co}\-efficients)
  satisfies~\cite{vbc,benedettipetronio}
  \[ \sv M \cdot \sv N 
     \leq \sv{M \times N}
     \leq {\dim M + \dim N \choose \dim N} \cdot \sv M \cdot \sv N.    
  \]
  The upper estimate is based on the homological cross-product and
  works in the same way also for the integral simplicial volume. But
  the lower estimate depends on the Hahn-Banach theorem, and the
  argument does not carry over to the integral setting. Hence, one
  might ask whether also $\isv M \cdot \isv N \leq \isv {M \times N}$
  holds.
\end{question}

\section{Integral simplicial volume is not computable}
\label{sec:noncomp}

We will now briefly introduce the relevant setup for computability and
derive the non-computability statement of Corollary~\ref{cor:noncomp}
from Theorem~\ref{thm:odd1}:

For $n \in \N$ we let $K_n$ be the countable set of all finite
$n$-dimensional simplicial complexes whose vertices form a subset
of~$\N$ and we let $T_n$ be the subset of all such complexes whose
geometric realisation is an oriented closed connected topological
$n$-manifold. We now consider an injection~$i_n \colon K_n
\longrightarrow \N$ with the property that for every~$K \in K_n$ the
number~$i_n(K)$ can be calculated by integer arithmetic from the set
of simplices, and conversely such that for every~$m \in i_n(K_n)$ the
complex~$K$ with~$i_n(K)=m$ can be reconstructed by integer
arithmetic. For example, such a map~$i_n$ can be obtained by
considering suitable products of prime numbers.

Using~$i_n$, we can interpret notions from computability on~$T_n$: A
function~$f \colon T_n \longrightarrow \N$ is \emph{computable}
[resp.\ \emph{semi-computable}], if the associated function~$ f \circ
i_n^{-1} \colon i_n(T_n) \longrightarrow \N$ is a partial recursive
[resp.\ \emph{$\mu$-recursive}] function~$\N \longrightarrow \N$ with
domain~$i_n(T_n)$. 
A subset~$A \subset T_n$ is called \emph{decidable}
[resp.\ \emph{semi-decidable}] if there is a computable
[resp.\ semi-computable] function~$f \colon T_n \longrightarrow \N$
satisfying~$A = f^{-1}(0)$ and $T_n \setminus A = f^{-1}(1)$. For the exact
definition of ($\mu$-)recursive functions and the relation with Turing
machines we refer to the literature~\cite{bbj}. 

We can now formulate Corollary~\ref{cor:noncomp} in a more precise way:

\begin{cor}[non-computability]
  Let $n \in \N_{\geq 5}$ be odd. 
  \begin{enumerate}
    \item 
      Then the set 
      \[ \bigl\{ K \in T_n \bigm| \isv{\, |K|\,} =1 \bigr\} 
      \]
      is \emph{not} decidable.
    \item
      In particular, the function
      \begin{align*}
        I_n \colon T_n & \longrightarrow \N \\
        K & \longrightarrow \isv{\,|K|\,}
      \end{align*}
      is \emph{not} computable.
  \end{enumerate}
\end{cor}
\begin{proof}
  \emph{Ad~1.}
  In view of
  Theorem~\ref{thm:odd1}, we have
  \[ \bigl\{ K \in T_n \bigm| \isv{\, |K| \,} =1 \bigr\} 
     = \bigl\{ K \in T_n \bigm| |K| \simeq S^n \bigr\},
  \]
  which is known to be \emph{not} decidable~\cite{markov}.

  \emph{Ad~2.}  If $I_n$ were computable, then the set $I_n^{-1}(1)
  \subset T_n$ were decidable, which contradicts the first part. 
  Hence, $I_n$ is \emph{not} computable.
\end{proof}

\begin{rem}[semi-computability]
  For all~$n \in \N$ the function~$I_n \colon T_n \longrightarrow
  \N$ is semi-computable in the following sense: If~$m \in \N$, then the
  set
  \[ I_n^{-1} \bigl(\{0,\dots,m\}\bigr) 
     = \bigl\{ K \in T_n 
       \bigm| \isv{\,|K|\,} \leq m \bigr\}
  \]
  is semi-decidable: Let $K \in T_n$. If $c \in C_n(|K|;\Z)$ is 
  a cycle, then the simplicial approximation theorem and an inductive 
  construction of (relative) homotopies shows that there is a~$k\in \N$ 
  and a singular chain~$c' \in C_n(|K|;\Z)$ with the following properties: 
  \begin{itemize}
    \item We have~$[c] = [c']$ in~$H_n(|K|;\Z)$ and $|c'|_1 \leq |c|_1$.
    \item The chain~$c'$ is a \emph{generalised simplicial chain},
      i.e., it consists of singular simplices~$\Delta^n
      \longrightarrow K$ that are geometric realisations of simplicial
      maps from the $k$-fold barycentric subdivision of~$\Delta^n$
      (viewed as simplicial complex in the standard way) to~$K$.
  \end{itemize}
  Moreover, all the following operations can be performed by Turing
  machines: enumerate all barycentric subdivisions of~$\Delta^n$,
  enumerate all simplicial maps between two finite simplicial
  complexes, check whether a generalised simplicial $\Z$-chain on a
  finite simplicial complex in~$T_n$ is a fundamental cycle, and
  compute the $\ell^1$-norm of a generalised simplicial singular $\Z$-chain.

  Hence, there is a Turing machine acting on inputs~$K$ from~$T_n$
  such that: If $\isv{|K|} \leq m$, then the Turing machine stops and
  accepts~$K$. If $\isv{|K|} > m$, then the Turing machine does not
  stop or, if it stops, does not accept~$K$.  In other words, the
  set~$I_n^{-1} (\{0,\dots,m\})$ is semi-decidable.
\end{rem}

However, in certain restricted situations, integral simplicial volume
at least is a finite-to-one map:

\begin{prop}
  Let $m \in \N$. Then there exist only finitely many homeomorphism types of 
  oriented closed connected $3$-manifolds~$M$
  with~$\isv M \leq m$.
\end{prop}

The proof uses models of singular cycles, similar to the ones
discussed in the previous sections. A more detailed description of
such combinatorial models of cycles and their associated homology
classes can be found in the discussion of weak
flexibility~\cite[Section~5.1]{ffsn_rep}.

\begin{proof}
  Let $n \in \N$. Then there exist only finitely many different
  combinatorial types of singular $n$-cycles with $\Z$-coefficients
  and $\ell^1$-norm at most~$m$. These finitely many combinatorial types 
  can be combined into a finite connected simplicial complex~$X$ and 
  a class~$\alpha \in H_n(X;\Z)$ with the following property: If $M$ 
  is an oriented closed connected $n$-manifold with~$\isv M \leq m$, then 
  there is a continuous map~$f \colon X \longrightarrow M$ satisfying
  \[ H_n(f;\Z)(\alpha) = \fcl M. 
  \]
  Now let $n = 3$; in particular, $n\leq 5$. Then by Thom's
  work~\cite{thom} on representing homology classes through manifolds,
  there is an oriented closed connected $n$-manifold~$N$ and a
  continuous map~$g \colon N \longrightarrow X$ with
  \[ H_n(g;\Z)(\fcl N) =\alpha. 
  \]
  Composing these maps, we see that $N$ dominates every oriented closed
  connected $3$-manifold~$M$ with~$\isv M \leq m$ through a map of
  degree~$1$. However, every oriented closed connected $3$-manifold
  can dominate only finitely many other homeomorphism types of
  oriented closed connected $3$-manifolds~\cite{liu,bhrw}.
\end{proof}



\medskip
\vfill

\noindent
\emph{Clara L\"oh}\\[.5em]
  {\small
  \begin{tabular}{@{\qquad}l}
    Fakult\"at f\"ur Mathematik\\
    Universit\"at Regensburg\\
    93040 Regensburg\\
    Germany\\
    \textsf{clara.loeh@mathematik.uni-regensburg.de}\\
    \textsf{http://www.mathematik.uni-regensburg.de/loeh}
  \end{tabular}}
\end{document}